\documentclass[a4paper,11pt]{amsart}
\usepackage[left=3.6cm, top=3cm, bottom=3.5cm, right=3.6cm]{geometry}
\usepackage{amsthm,amsmath}
\usepackage{amssymb}
\usepackage{mathtools}
\usepackage{enumitem}
\usepackage{graphicx}
\usepackage{microtype}
\usepackage[T1]{fontenc}

\newtheorem{Thm}{Theorem}

\newtheorem{Lem}[Thm]{Lemma}
\numberwithin{Thm}{section}

\theoremstyle{definition}
\newtheorem{Def}[Thm]{Definition}
\newtheorem{Rem}[Thm]{Remark}

\def\EE{\mathcal{E}}
\def\WW{\mathcal{W}}
\def\N{\mathbb{N}}
\def\sE{\mathcal{E}}

\def\HH{\mathcal{H}}

\def\DD{\mathcal{D}}
\def\PP{\mathbf{P}}
\def\CC{\mathbf{C}}
\def\RR{\mathbf{R}}
\def\C{\mathbf{C}}

\usepackage{nicefrac}
\makeatletter
\newsavebox\myboxA
\newsavebox\myboxB
\newlength\mylenA
\newcommand*\xoverline[2][0.75]{%
    \sbox{\myboxA}{$\m@th#2$}%
    \setbox\myboxB\null
    \ht\myboxB=\ht\myboxA%
    \dp\myboxB=\dp\myboxA%
    \wd\myboxB=#1\wd\myboxA
    \sbox\myboxB{$\m@th\overline{\copy\myboxB}$}
    \setlength\mylenA{\the\wd\myboxA}
    \addtolength\mylenA{-\the\wd\myboxB}%
    \ifdim\wd\myboxB<\wd\myboxA%
       \rlap{\hskip 0.5\mylenA\usebox\myboxB}{\usebox\myboxA}%
    \else
        \hskip -0.5\mylenA\rlap{\usebox\myboxA}{\hskip 0.5\mylenA\usebox\myboxB}%
    \fi}
\makeatother

\title{Exotic holomorphic Engel structures on $\CC^4$}
\author{R. Coelho}
\address{Mathematisches Institut, LMU M\"unchen, Theresienstr. 39, 80333 M\"unchen, Germany}
\email{mcoelho@math.lmu.de}
\author{N. Pia}
\address{Dipartimento di Matematica e Informatica, Universit\`a degli Studi di Cagliari, Palazzo delle Scienze, Via Ospedale 72, 09124 Cagliari, Italy}
\email{nicola.pia@unica.it}

\begin{document}
%
%
\begin{abstract}

A holomorphic Engel structure determines a flag of distributions $\WW\subset \DD\subset \EE$. 
We construct examples of Engel structures on $\CC^4$ such that each of these distributions is hyperbolic in the sense that it has no tangent copies of $\CC$. 
We also construct two infinite families of pairwise non-isomorphic Engel structures on $\CC^4$ by controlling the curves $f:\CC\to \CC^4$ tangent to $\WW$.
The first is characterised by the topology of the set of points in $\CC^4$ admitting $\WW$-lines, and the second by a finer geometric property of this set. 
A consequence of the second construction is the existence of uncountably many non-isomorphic holomorphic Engel structures on $\CC^4$.
\end{abstract}
\maketitle
%
%
\section{Introduction}
A \emph{holomorphic Engel structure} on a complex manifold $M$ of complex dimension $4$ is a holomorphic subbundle $\DD\hookrightarrow TM$ of complex rank $2$ which is maximally non-integrable.
More precisely $[\DD,\DD]=\EE$ has constant rank $3$ and satisfies $[\EE,\EE]=TM$ (a 3-distribution satisfying this condition is called a \emph{holomorphic even contact structure}).
Every holomorphic even contact structure $\EE$ admits a unique holomorphic line field $\WW\subset\EE$ such that $[\WW,\EE]\subset\EE$.
This line field $\WW$ is called \emph{characteristic line field}.
If $\DD$ is an Engel structure and $\EE=[\DD,\DD]$ is its associated even contact structure then the characteristic line field $\WW$ satisfies $\WW\subset\DD$.
Hence an Engel structure $\DD$ determines a flag of distributions $\WW\subset\DD\subset\EE$.

Every holomorphic Engel structure $(M,\DD)$ is locally isomorphic to the complex Euclidean space $\CC^4$ with coordinates $(w,x,y,z)$ and the Engel structure given by
\[
\DD_{st} = \ker(dy - z dx)\cap \ker(dz-wdx).
\]
The associated even contact structure is $\EE_{st}=\ker(dy-zdx)$ and the characteristic line field is $\WW_{st}=\ker(dx\wedge dy\wedge dz)$.

These structures are the holomorphic analogues of the usual Engel structures.
Together with line fields, contact structures and even contact structures, these are the only topologically stable distributions (see \cite{cartan}).
The existence of an orientable Engel structure on a closed orientable (real) 4-manifold $M$ implies that $M$ is parallelizable. 
Conversely the existence of Engel structures on parallelizable 4-manifolds was established in \cite{vogel}.
The geometry of these structures is closely related to even contact structures, which are known to satisfy a complete $h$-principle (see \cite{mcduff}).
An existence $h$-principle has been established for Engel structures in \cite{cp3}.

Holomorphic Engel structures on closed complex $4$-manifolds have been studied in \cite{presas}.
The only known constructions are the Cartan prolongation of a holomorphic contact structure and the Lorentz tube of a holomorphic conformal structure on a $3$-manifold.
These two families of structures are classified in the projective case, and the main result in \cite{presas} is a partial classification of Engel structures on closed projective manifolds.
The existence of a holomorphic Engel structure on a closed complex manifold which is not a Cartan prolongation or a Lorentz tube remains an open problem.

We are interested in constructing non-standard holomorphic Engel structures on $\CC^4$.
Forstneri\v{c} constructed non-standard holomorphic contact structures on $\CC^{2n+1}$ in \cite{forstneric}.
There the idea is to find a Fatou-Bieberbach domain where the standard holomorphic contact structure is \emph{hyperbolic} in a directed sense, as explained below.
One of the aims of this note is to use the same method to prove the analogous statement for holomorphic Engel structures.
In what follows, given a distribution $\HH\to T\CC^4$, we will use the terms \emph{$\HH$-line} or \emph{line tangent to $\HH$} to designate a non-constant holomorphic map $f:\CC\to\CC^4$ such that $f'(\zeta)\in\HH_{f(\zeta)}$ for all $\zeta\in\CC$.
If no ambiguity concerning the distribution may arise, we also use \emph{horizontal line} as a synonym for $\HH$-line.

\begin{Thm}\label{mainthm}
On $\CC^4$ there are Engel structures $\DD_\EE$, $\DD_\DD$ and $\DD_\WW$ with the following properties
\begin{enumerate}
 \item $\DD_\EE$ admits no lines tangent to its associated even contact structure;
 \item $\DD_\DD$ admits no $\DD_\DD$-lines but does admit lines tangent to its associated even contact structure;
 \item $\DD_\WW$ admits no lines tangent to its characteristic foliation but does admit $\DD_\WW$-lines.
\end{enumerate}
In particular these Engel structures are pairwise non-isomorphic and not isomorphic to the standard Engel structure $(\CC^4,\DD_{st})$. 
\end{Thm}
As we verify below, the standard Engel structure admits many $\DD_{st}$-lines, including many tangent to the characteristic foliation.

Controlling the geometry of the characteristic foliation, we are able to construct infinite families of non-isomorphic holomorphic Engel structures.

\begin{Thm}\label{mainthm2}
For every $n\in\N\cup\{\infty\}$ there exists an Engel structure $\DD_n$ on $\CC^4$ for which the only $\DD_n$-lines are tangent to the characteristic foliation $\WW_n$, and such that
\[
 L_n:=\{p\in\CC^4:\exists f:\C\to\CC^4\ \DD_n\text{-line with }f(0)=p\}
\]
is a proper subset of $\CC^4$ which has exactly $n$ connected components for $n\in\N$, and $L_\infty=\CC^4$.
\end{Thm}

We first construct $\DD_\infty$ using an open set in the Cartan prolongation of a Kobayashi hyperbolic contact structure in $\CC^3$.
This will admit very few $\DD_\infty$-lines by construction.
Then we use a result, due to Buzzard and Forn\ae ss (theorem \ref{buzz}, for a proof see \cite{buzzard}) that allows one to control the set of points in $\CC^4$ which admit such horizontal lines.
A more careful analysis leads to
\begin{Thm}\label{mainthm3}
For every $R\in\RR\setminus\{0\}$ there exists an Engel structure $\DD_{R}$ for which the only  $\DD_{R}$-lines are tangent to the characteristic foliation $\WW_{R}$, and such that the set of points which admit such $\WW_{R}$-lines is exactly $\CC\times\{0,1,R\sqrt{-1}\}\times\CC^2\subset\CC^4_{(w,x,y,z)}$.
Moreover $\DD_{R}$ is isomorphic to $\DD_{R'}$ if and only if $R=R'$.
\end{Thm}

%
%

\section{Hyperbolicity and holomorphic Engel structures}
For the proof of theorem \ref{mainthm} we will need the notion of hyperbolicity on a complex directed manifold.
Recall that the Kobayashi pseudo-distance $d_M$ on a complex manifold $M$ may be written in terms of the Finsler pseudo-metric
\begin{equation}\label{finsler}
F(v_p) = \inf\left\lbrace \frac{1}{|\lambda|}:\exists\text{ a holomorphic }f:D\to M\text{ s.t. }f(0)=p,\,f'(0)=\lambda v \right\rbrace,
\end{equation}
by integration. Explicitly, 
\begin{equation}\label{integrated}
d(p,q) = \inf\left\lbrace \int_0^1 F(\gamma'(t))dt:\gamma \text{ piecewise smooth, }\gamma(0)=p\text{ and }\gamma(1)=q\right\rbrace.
\end{equation}
Given a holomorphic subbundle $\HH \subset TM$, a disc $D\to M$ is called \emph{horizontal} if it is tangent to $\HH$.
The Finsler pseudo-metric $F_\HH$ directed by $\HH$ is defined by requiring that the infimum in (\ref{finsler}) be taken only over horizontal discs. Likewise, the Kobayashi pseudo-distance $d_\HH$ on the directed manifold $(M,\HH)$ is defined by requiring that the infimum in (\ref{integrated}) be taken only over paths $\gamma$ that are tangent to $\HH$.
This is finite because, by Chow's theorem, these paths always exist if the distribution is bracket generating. This is the case, by definition, for an Engel structure.
The directed manifold $(M,\HH)$ is said to be \emph{Kobayashi hyperbolic} if $d_\HH$ is a genuine distance. 
Note that if $(M,\HH)$ is Kobayashi hyperbolic, there can be no $\HH$-line.

\begin{Rem}
Notice that the standard Engel structure is not hyperbolic, since it admits many horizontal lines $f:\CC\to\CC^4$.
For instance, one can take the leaves of the characteristic foliation $\WW$ of $\DD_{st}$.
In fact, given a point $p=(w_0,x_0,y_0,z_0)\in \CC^4$ and a vector $v=(v_w,v_x,v_y,v_z)\in\DD_p$ (hence $v_z=w_0 v_x$ and $v_w = z_0 v_x$)
the map
\[
f(\zeta)=\left(w_0+v_w\zeta,x_0+v_x\zeta, y_0+v_y\zeta+v_x v_z \frac{\zeta^2}{2} + v_x^2 v_w \frac{\zeta^3}{6},z_0+v_z\zeta + v_x v_w \frac{\zeta^2}{2}\right)
\]
is a horizontal line with $f(0)=p$ and $f'(0)=v$.
\end{Rem}

The idea for proving theorem \ref{mainthm} is to construct certain (directed) hyperbolic subsets of $\C^4$ and look for biholomorphic copies of $\CC^4$ inside these domains.

\begin{Def} A Fatou-Bieberbach domain is a proper subset $\Omega\subset \CC^n$ such that $\Omega$ is biholomorphic to $\CC^n$.
\end{Def}

Following \cite{forstneric} we let $\{c_n\}_{n\in\mathbb{N}}$, $\{d_n\}_{n\in\mathbb{N}}$  and $\{e_n\}_{n\in\mathbb{N}}$ be positive diverging monotonic sequences.
Denote with $D_y$ (resp. $D_z$) the unit disc in the $y$ (resp. $z$) direction, with $\partial D^2_{(w,x)}$ the boundary of the unit polydisc in the $(w,x)$-plane in $\CC^4$ and with $\partial D^3_{(w,x,z)}$ the boundary of the unit polydisc in the $(w,x,z)$-plane in $\CC^4$.
Let
\begin{equation}\label{defA} 
A=\bigcup_{i=1}^\infty 2^{i-1}\partial D^3_{(w,x,z)}\times c_i \xoverline{D}_y.
\end{equation}
\begin{equation}\label{defB} 
B=\bigcup_{i=1}^\infty 2^{i-1}\partial D^2_{(w,x)}\times d_i \xoverline{D}_y\times e_i\xoverline{D}_z.
\end{equation}

By a direct adaptation of lemma 2.1 in \cite{forstneric}, we can prove the following:

\begin{Lem}\label{tlem}
Assume $d_n\geq 2^{5n+2}$ and $e_n\geq 2^{3n+1}$ for every $n\in \mathbb{N}$. Let $N_0\in \mathbb{N}$ and $f:D\to \CC^4\setminus B$ be a $\DD_{st}$-horizontal embedding of a disc with $f(0)\in 2^{N_0} D^4$. Then we have the estimates
\[ 
|w'(0)|<2^{N_0+1},\,\,|x'(0)|<2^{N_0+1},\,\,|y'(0)|<2^{3N_0+2},\,\,|z'(0)|<2^{2N_0+1}.
\]
\end{Lem}
\begin{proof}
We may assume without loss of generality that $f$ is holomorphic on $\xoverline D$ (replace $f$ by $\zeta\mapsto f(r\zeta)$ for some $r<1$). This gives $N\in \mathbb{N}$ such that $|x(\zeta)|<2^N$ and
$|w(\zeta)|< 2^N$ for all $\zeta\in \xoverline{D}$. 
The Cauchy integral formula for a circle centered at $\zeta=0$ of ray $r=1-2^{-N}$ gives
\[
|x'(\zeta)|<2^{2N}\quad\text{and}\quad|w(\zeta)x'(\zeta)|<2^{3N}
\]
for $|\zeta|\leq r$.
Since $f$ is horizontal, we have the conditions
\begin{equation}\label{sol}
y'(\zeta)=z(\zeta)x'(\zeta)\quad\text{and}\quad z'(\zeta)=w(\zeta)x'(\zeta)
\end{equation}
which in turn give
\begin{align*}
|z(\zeta)|&\leq |z(0)|+\left| \int_0^\zeta w dx \right| <2^{N_0}+2^{3N}<2^{3N+1}\leq d_N\\
|y(\zeta)|&\leq |y(0)|+\left| \int_0^\zeta z dx \right| <2^{N_0}+2^{5N+1}<2^{5N+2}\leq c_N
\end{align*}
for $|\zeta|\leq r$. 
From these estimates, the definition of $B$, and the fact that $f(D)$ does not intersect $B$, it follows that $(w(\zeta),x(\zeta))$ does not intersect $2^{N-1}\partial D^2$ for $|\zeta|\leq r$. Since $2^{N-1}\partial D^2$ disconnects $2^N D^2$ and $(w(0),x(0))\in 2^{N_0}D^2\subset 2^{N-1}D^2$, we conclude that
\[
(w(\zeta),x(\zeta))\in 2^{N-1} D^2\quad\text{for}\quad|\zeta|\leq 1-2^{-N}.
\]
If $N-1>N_0$, we can repeat the same argument to get
\[
(w(\zeta),x(\zeta))\in 2^{N-2} D^2\quad\text{for}\quad|\zeta|\leq 1-2^{-N}-2^{-(N-1)},
\]
and after finitely many repetitions
\[
(w(\zeta),x(\zeta))\in 2^{N_0} D^2\quad\text{for}\quad|\zeta|\leq 1-2^{-N}-\ldots-2^{-(N_0+1)}\leq \frac{1}{2}.
\]
Applying the Cauchy estimate now gives $|x'(0)|\leq 2^{N_0+1}$ and $|w'(0)|\leq 2^{N_0+1}$, while using equation (\ref{sol}) we get
\[
|z'(0)|=|w(0)x'(0)|\leq 2^{2N_0+1}\quad\text{and}\quad|y'(0)|=|z(0)x'(0)|\leq 2^{3N_0+2},
\]
completing the proof of the lemma.
\end{proof}

The following lemma has a completely analogous proof.
\begin{Lem}\label{tlem2}
Assume $c_n\geq 2^{3n+1}$ for every $n\in \mathbb{N}$. Let $N_0\in \mathbb{N}$ and $f:D\to \CC^4\setminus A$ be a $\DD_{st}$-horizontal embedding of a disc with $f(0)\in 2^{N_0} D^4$. Then we have the estimates
\[ 
|w'(0)|<2^{N_0+1},\,\,|x'(0)|<2^{N_0+1},\,\,|y'(0)|<2^{2N_0+1},\,\,|z'(0)|<2^{N_0+1}.
\]
\end{Lem}

The following theorem was proved by Forstneri\v{c} in \cite{forstneric}.
\begin{Thm}[Forstneri\v{c}]\label{fosThm}
Let $0<a_1<b_1<a_2<b_2<\ldots$ and $c_i>0$ be sequences of real numbers such that $\lim_{n\to\infty}a_n=\lim_{n\to\infty}b_n=+\infty$.
Let $n>1$ be an integer and
\begin{equation}\label{defK}
 K=\bigcup_{i=1}^\infty \left(b_i\xoverline{D}^{n-1}\setminus a_iD^{n-1}\right)\times c_i \xoverline D\subset \CC^n.
\end{equation}
Then there exists a Fatou-Bieberbach domain $\Omega\subset\CC^n\setminus K$.
\end{Thm}

%
%
\section{Proof of theorem \ref{mainthm}}

In what follows fix $0<\varepsilon<1$ and consider the real sequences 
\[
a_i=2^{i-1}-\varepsilon, \,\,b_i=2^{i-1}+\varepsilon.
\]

To construct $\DD_\EE$ we fix $c_i=2^{3i+1}$ and let $A$ be the set determined by $c_i$ according to \eqref{defA}.
Lemma \ref{tlem2} ensures that $(\CC^4\setminus A,\EE_{st})$ is hyperbolic, moreover theorem \ref{fosThm} gives a Fatou-Bieberbach map $\Phi:\C^4\to\Omega\subset\CC^4\setminus A$.
We set $\DD_\EE:=\Phi^\ast\DD_{st}$ so that its associated even contact structure is $\Phi^\ast\EE_{st}$.
Lemma \ref{tlem2} furnishes a lower bound for the Finsler metric, whence it follows that the $\Phi^\ast\EE_{st}$-directed Kobayashi pseudo-distance on $\Omega$ is a genuine distance, i.e. the restriction of the standard even contact structure to $\Omega$ is hyperbolic.

To construct $\DD_\DD$ we fix $d_i=2^{5i+2}$ and $e_i=2^{3i+1}$ and let $K$ be the set determined by $n=3, a_i, b_i$ and $c_i=d_i$ according to \eqref{defK}.
Let $B$ be the set determined by $d_i$ and $e_i$ according to \eqref{defB}, and notice that $B\subset K\times \CC$.
By theorem \ref{fosThm} there exists a Fatou-Bieberbach domain $\Omega\subset \CC^3$ with $\Omega \cap K=\emptyset$. Define $\Xi=\Omega\times \CC$. The subset $\Xi\subset\CC^4$ is a Fatou-Bieberbach domain in $\CC^4$ which fulfills $\Xi\cap \left(K\times \CC\right)=\emptyset$; in particular, $\Xi\cap B=\emptyset$. 
Let $\Phi:\C^4\to\Xi$ be the Fatou-Bieberbach map. 
We define $\DD_\DD=\Phi^\ast(\DD_{st})$.
Lemma \ref{tlem} furnishes a lower bound for the Finsler metric, whence it follows that the $\DD_{st}$-directed Kobayashi pseudo-distance on $\Xi$ is a genuine distance, i.e. the restriction of the standard Engel structure to $\Xi$ is hyperbolic.
Notice that in this construction the associated even contact structure $\sE$ is not hyperbolic.
Indeed we have many $\sE_{st}$-lines $f:\CC\to\Xi$ of the form
\[
f(\zeta)=(w_0,x_0,y_0,\zeta)
\]
where $(w_0,x_0,y_0)$ is not contained in $A$, which can be pulled-back.

To construct $\DD_\WW$ consider the set
\[
K=\bigcup_{i=1}^\infty2^{i-1}\partial  D_{(w,y)}^2\times 2^i \xoverline{D}_z
\]
contained in the $(w,y,z)$-plane in $\CC^4$.
All $\WW$-horizontal holomorphic copies of $\CC$ are of the form $f(\zeta)=(w(\zeta),x_0,y_0,z_0)$ for some $w$ holomorphic and hence they will intersect $K$ for some $\zeta$.
Indeed if $N_0\in\N$ is such that $|z_0|<d_{N_0}$ then $f$ does not intersect $K$ only if $|w(\zeta)|<2^{N_0-1}$ for all $\zeta\in\C$, which is not true.
Theorem \ref{fosThm} ensures the existence of a Fatou-Bieberbach map $\tilde\Phi:\C^3\to\Omega\subset\C^3\setminus K$ so that also $\Phi=\tilde\Phi\times id:\CC^4\to\Omega\times\C\subset\C^4$ is a Fatou-Bieberbach map.
By the above discussion there are no copies of $\CC$ tangent to the characteristic foliation of the standard Engel structure restricted to $\Omega$.
We then define $\DD_\WW:=\Phi^*\DD_{st}$, this structure does not have lines tangent to the characteristic foliation, nevertheless $\CC^4$ is not $\DD_\WW$-hyperbolic, since the pull-back of the $\DD_st$-line 
\[
f:\C\hookrightarrow\C^4\qquad f(\zeta)=(0,\zeta,0,0).
\]
is a $\DD_\WW$-line.

%
%

\section{Construction of the infinite families}\label{example}
In this section we will prove theorems \ref{mainthm2} and \ref{mainthm3}.

\subsection{Proof of theorem \ref{mainthm2}}

We use Forstneri\v{c}'s hyperbolic contact structure on $\CC^3$, which is the pull-back $\alpha=\Phi^*\alpha_{st}$ of the restriction of the standard contact structure on a hyperbolic Fatou-Bieberbach domain in $\CC^3\setminus K$ (see \cite{forstneric}).
Consider the Cartan prolongation $M=\PP(\xi_h)$ of $\xi_h=\ker\alpha$ with its Engel structure $\DD(\xi_h)$. Since $\ker\alpha_{st}$ is trivial as a holomorphic bundle, $M$ is biholomorphic to $\CC^3\times\CC\PP^1$.
Given $p\in\CC\PP^1$, consider in $M$ the open set $\CC^4=\CC^3\times\CC\PP^1\setminus(\CC^3\times\{p\})$ and the restriction of the Engel structure $\DD_\infty=\left.\DD(\xi_h)\right|_{\CC^4}$.
We claim that this structure has the properties stated in theorem \ref{mainthm2}.

Indeed suppose that $f:\C\to\CC^4$ is a $\DD_\infty$-line. 
Then if we denote by $\pi:M\to\CC^3$ the canonical projection of the projectivisation, the composition $\pi\circ f$ is tangent to $\xi_h$ in $\CC^3$.
Since $(\CC^3,\xi_h)$ is hyperbolic, $\pi\circ f$ must be constant, so $f$ is tangent to the fibers. This proves that the only $\DD_\infty$-lines are tangent to the characteristic foliation $\WW_\infty$.

Fix $n\in\N$. 
In order to construct $\DD_n$, we use the following result
\begin{Thm}[Buzzard and Forn\ae ss, \cite{buzzard}]\label{buzz}
 Let $L$ be a closed, $1$-dimensional, complex subvariety of $\CC^2$, and $B_0$ a ball with $\xoverline B_0\cap L=\emptyset$.
 Then there exists a Fatou-Bieberbach domain $\Omega\subset \CC^2\setminus\xoverline B_0$ with $L\subset \Omega$ and a biholomorphic map $\Psi$ from $\Omega$ onto $\CC^2$ such that $\CC^2\setminus \Psi(L)$ is Kobayashi hyperbolic.
 Moreover, all nonconstant images of $\CC$ in $\CC^2$ intersect $\Psi(L)$ in infinitely many points.
\end{Thm}

Now we choose
\[
\tilde L_n=\bigcup_{k=1}^n\CC\times\{k\}\subset\CC^2_{(w,x)}\,\,.
\]
Then theorem \ref{buzz} gives a Fatou-Bieberbach map $\Phi_n:\CC^2\to\Omega_n\subset\CC^2$ such that $\Omega_n\setminus \tilde L_n$ is Kobayashi hyperbolic and the $w$-curves $f_i:\C\to\C^2$ s.t. $\zeta\mapsto(\zeta,i)$ are still contained in $\Omega_n$.
Now take the Fatou-Bieberbach map $\Psi_n=\Phi_n\times id:\CC^4\to\Omega_n\times\CC^2\subset\CC^4$ and the Engel structure $\DD_n=\Psi_n^*\DD_\infty$.
By construction $\DD_n$ only admits $\DD_n$-lines on the points
\[
L_n=\tilde L_n\times\CC^2=\big\{(w,x,y,z)\in\CC^4:x\in\{1,...,n\}\big\} 
\]
hence completing the proof of theorem \ref{mainthm2}.

\subsection{Proof of theorem \ref{mainthm3}}
For some $R\in\mathbf{R}\setminus\{0\}$, we will consider the subvariety $C_R=\left(\CC\times\{0,1,R\sqrt{-1}\}\right)\cup\left(\{0\}\times\CC\right)\subset \CC^2$.
By theorem \ref{buzz}, there exists a Fatou-Bieberbach domain $\Omega_R\subset \CC^2$ which contains $C_R$, and such that the complement $\Omega_R\setminus C_R$ is Kobayashi hyperbolic. Moreover, any curve $\CC\to\Omega_R$ intersects $C_R$ an infinite number of times.
Denote by $\WW_{R}$, resp. $\WW_{R'}$, the $1$-foliation on $\Omega_R\times \CC^2 $, resp. $\Omega_R'\times\CC^2$, determined by the projections $p:\Omega_R\times\CC^2\to \CC^3$, resp. $p':\Omega_{R'}\times\CC^2\to \CC^3$, given by $(w,x,y,z)\mapsto (x,y,z)$. 
We introduce also the projections $\pi:\Omega_R\times\CC^2\to \CC$ and $\pi':\Omega_{R'}\times\CC^2\to \CC$ given by $(w,x,y,z)\mapsto x$ and the notation $V_R=\pi^{-1}\{0,1,R\sqrt{-1}\}$ and $V'_{R'}={\pi'}^{-1}\{0,1,R'\sqrt{-1}\}$.
Notice that $V_R$, resp. $V'_{R'}$, consists exactly of the points of $\Omega_R$, resp. $\Omega_{R'}$ through which a $\WW_{R}$-line, resp. $\WW_{R'}$-line, passes.

\begin{Lem}\label{notbiholomorphic}
Suppose that $R,R'\in\mathbf{R}\setminus\{0\}$ and $R\neq R'$. 
Then there exists no biholomorphic map $\Phi:\Omega_R\times \CC^2 \to\Omega_{R'}\times \CC^2$ such that $\Phi_\ast(\WW_{R})=\WW_{R'}$.
\begin{proof}
Suppose such a $\Phi$ exists and consider the map $h:\CC\to\CC$ given by $h=\pi'\circ \Phi \circ \iota$, where $\iota$ is the inclusion $\iota(\zeta)=(0,\zeta,0,0)\in \Omega_R\times \CC^2$.
Notice that horizontal curves in $\WW_{R}$ must map to horizontal curves in $\WW_{R'}$.
Moreover, we have $h^{-1}\{0,1,R'\sqrt{-1}\}=\{0,1,R\sqrt{-1}\}$.
It follows that we have a biholomorphic map $\left.\Phi\right|_{V_R}:V_R\to V'_{R'}$.
This implies in particular that $h:\{0,1,R\sqrt{-1}\}\to\{0,1,R'\sqrt{-1}\}$ is bijective.
Since $h$ is non-constant, it either has an essential singularity or a pole at infinity.

If $h$ has an essential singularity at infinity, then by the big Picard theorem $h$ takes every value in $\CC$ infinitely many times, with one possible exception.
This contradicts the fact that $h:\{0,1,R\sqrt{-1}\}\to\{0,1,R'\sqrt{-1}\}$ is bijective. 

Otherwise, $h$ is a polynomial with exactly one zero, so it must be linear.
On the other hand, $h(\{0,1,R\sqrt{-1}\})=\{0,1,R'\sqrt{-1}\}$, which is impossible for $R\neq R'$.
\end{proof}

\end{Lem}

Now given the Fatou-Bieberbach map $\Phi_R:\CC^4\to\Omega_R\times\CC^2\subset\CC^4$ we define $\DD_{R}:=\Phi_R^*\DD_{st}$ and theorem \ref{mainthm3} is a direct consequence of lemma \ref{notbiholomorphic}.

\subsection*{Acknowledgements.} We are grateful to our advisor D. Kotschick for proposing this problem to us, and for several stimulating discussions. 
The example in the beginning of section \ref{example}, which is the starting point for theorems \ref{mainthm2} and \ref{mainthm3}, was suggested by him.


\begin{thebibliography}{99}
\bibitem{buzzard}
	G. T. Buzzard and J. E. Forn\ae ss,
	\emph{An embedding of $\CC$ in $\CC^2$ with hyperbolic complement}\\
	Math. Ann. 306 (1996), 539-546. 


\bibitem{cartan}
	E. Cartan,
	\emph{Sur quelques quadratures dont l'\'el\'ement diff\'erentiel contient des fonctions arbitraires}\\
	Bull. Soc. Math. France 29 (1901), 118-130. 


\bibitem{cp3}
	R. Casals, J.L. Perez, A. del Pino and F. Presas,
	\emph{Existence $h$-Principle for Engel structures}\\
	Invent. math. (2017), DOI:10.1007/s00222-017-0732-6.


\bibitem{forstneric}
	F. Forstneri\v{c},
	\emph{Hyperbolic complex contact structures on $\CC^{2n+1}$}\\
	J. Geom. Anal. (2017), DOI:10.1007/s12220-017-9800-9.


\bibitem{mcduff}
	D. McDuff,
	\emph{Applications of convex integration to symplectic and contact geometry}\\
	Ann. Inst. Fourier 37 (1978), 107-133.	


\bibitem{presas}
	F. Presas and L.E. Sol\'a Conde,
	\emph{Holomorphic Engel structures}\\
	Rev. Mat. Complut. (2014) 27: 327. 


\bibitem{vogel}
	T. Vogel,
	\emph{Existence of Engel structures}\\
	Ann. of Math. (2)  169 (2009), no. 1, 79-137.
  
\end{thebibliography}
\end{document}